\documentclass{article}
\usepackage{amsmath}
\usepackage{amsthm}
\usepackage{amssymb}
\usepackage{array}
\usepackage{bbm}
\usepackage{booktabs}
\usepackage{caption}
\usepackage{enumerate}
\usepackage{graphicx}
\usepackage[colorlinks,linkcolor=red,anchorcolor=green,citecolor=blue]{hyperref}
\usepackage{indentfirst}
\usepackage[utf8]{inputenc}
\usepackage{listings}
\usepackage{mathrsfs}   
\usepackage{setspace}
\begin{document}
\theoremstyle{definition} 
    \newtheorem{fact}{Fact}
    \newtheorem{pro}{Question}
    \newtheorem{thm}{Theorem}[section]
    \newtheorem*{mainthm}{Main Theorem}
    \newtheorem{thmIntro}{Theorem}
    \renewcommand{\thethmIntro}{\Alph{thmIntro}}
    \newtheorem{coroIntro}[thmIntro]{Corollary}
    \newtheorem{lem}[thm]{Lemma}
    \newtheorem{coro}[thm]{Corollary}
    \newtheorem{ex}[thm]{Example}
    \newtheorem{defi}[thm]{Definition}
    \newtheorem{prop}[thm]{Proposition}
    \newtheorem{property}[thm]{Property}
    \newtheorem{rmk}[thm]{Remark}
    \newtheorem*{thme}{THEOREM}

    \newcommand{\vphi}[0]{\varphi}
    \newcommand{\upst}[0]{^{*}}
    \newcommand{\dwst}[0]{_{*}}
    \newcommand{\dd}[1]{\textrm{d}#1}
    \newcommand{\metr}{\textrm{g}}
    \newcommand{\Metr}{\textrm{G}}
    \newcommand{\rr}{\textrm{R}}
    \newcommand{\rmnum}[1]{\romannumeral #1}
    \newcommand{\mR}{\mathbb{R}}
    \newcommand{\mC}{\mathbb{C}}
    \newcommand{\mH}{\mathbb{H}}
    \newcommand{\mD}{\mathbb{D}}
    \newcommand{\mS}{\mathbb{S}}
    \newcommand{\mZ}{\mathbb{Z}}
    \newcommand{\z}{\zeta}
    \newcommand{\tphi}{\tilde{\phi}}
    \newcommand{\cp}{\mathbb{CP}^1}
    \newcommand{\sm}{\mathcal{M}}
    \newcommand{\warp}{\times_{f}}
    \newcommand{\wrp}[1]{\times_{#1}}
    \newcommand{\warppro}{M\warp N}
    \newcommand{\DF}{F^{'}}
    \newcommand{\mHi}[1]{\mH^n_{(#1)}}
    \newcommand{\gi}[1]{\text{g}_{(#1)}}
    \newcommand{\uppri}{^{'}}
    \newcommand{\upcir}{^{\circ}}
    \newcommand{\inpro}[2]{\langle#1,#2\rangle}
    \newcommand{\diff}{\text{d}}
    \newcommand{\pp}[1]{\frac{\partial}{\partial #1}}
    \newcommand{\tg}[2]{\tilde{\metr}\left(#1,#2\right)}
    \newcommand{\tn}[2]{\tilde{\nabla}_{#1}#2}
    \newcommand{\tmf}{\tilde{\mathcal{F}}}
    \newcommand{\codim}{\operatorname{codim}}
    \newcommand{\F}{\mathcal{F}}
    \newcommand{\nsb}[1]{\text{NS}(#1)}
    \newcommand{\myeqref}[1]{\text{(}\ref{#1}\text{)}}
    \newcommand{\inv}[1]{\operatorname{Inv}_0(#1)}
    \newcommand{\invre}[1]{\operatorname{Inv}^{-}_0(#1)}
    \newcommand{\homeo}[1]{\operatorname{Heomo}\left(#1\right)}
    \newcommand{\homeoori}[1]{\operatorname{Heomo}^{+}\left(#1\right)}
    \newcommand{\isotoid}[1]{\operatorname{Heomo}_0\left(#1\right)}
    \newcommand{\MCG}[1]{\operatorname{Mod}\left(#1\right)}
    \newcommand{\eMCG}[1]{\operatorname{Mod}^{\pm}\left(#1\right)}

    \newtheorem{thmintro}{Theorem}
    \renewcommand{\thethmintro}{\Alph{thmintro}}
    \newtheorem{corointro}[thmintro]{Corollary}

\title{The Equivalence of the Existences of Transnormal and Isoparametric Functions on Compact Manifolds}


\author{Minghao Li, Ling Yang}
\date{ }

\renewcommand{\thefootnote}{}
\footnotetext{\textit{Email addresses}: mhli19@fudan.edu.cn (Minghao Li), yanglingfd@fudan.edu.cn (Ling Yang).}

\maketitle
\renewcommand{\proofname}{\bf Proof.}

\begin{abstract}
Through exploring the embedded transnormal systems of codimension 1, we show the existence of a transnormal function 
on a connected complete Riemannian manifold requires the underlying manifold to have a vector bundle structure or a linear double disk bundle decomposition.
Conversely, any smooth manifold with either of these structures can be endowed with a Riemannian metric so that it admits a transnormal function, which, under suitable compactness conditions, can become isoparametric. As a corollary, for compact manifolds, the existences of transnormal and isoparametric functions impose the same topological constraints.

\end{abstract}


\section{Introduction}

A smooth function $f$ on a Riemannian manifold $(M,g)$ is called {\it transnormal} if there is a smooth function $b$ such that $|\nabla f|^2=b(f)$. If, in addition, $\Delta f=a(f)$ for another smooth function $a$, then $f$ is called {\it isoparametric}. Each regular level hypersurface of an isoparametric function is called an {\it isoparametric hypersurface}. The classification of isoparametric hypersurfaces in certain special spaces has a long and rich history. The study of isoparametric hypersurfaces in space forms was notably advanced by Cartan (\cite{Cartan1,Cartan2}), who completed the classification in Euclidean spaces $\mR^n$ and hyperbolic spaces $\mH^n$ of arbitrary dimensions and initiated the investigation of the classification in spheres $\mS^n$. The latter problem was found to be particularly challenging, as evidenced by the discovery of examples of inhomogeneous isoparametric hypersurfaces constructed in \cite{ozeki1975some,ferus1981cliffordalgebren}. After considerable effort spanning several decades, the classification of isoparametric hypersurfaces in spheres was ultimately completed through the contributions of several mathematicians, as documented in works such as \cite{takagi1972principal,munzner1980isoparametric,munzner1981isoparametric,abresch1983isoparametric,dorfmeister1985isoparametric,cecil2007isoparametric,immervoll2008classification,chi2011isoparametric,chi2013isoparametric,miyaoka2013isoparametric,miyaoka2016errata,chi2020isoparametric}. Meanwhile, the classification problem has also been extended to other special manifolds, such as complex projective spaces $\mC P^n$ (\cite{wang1982isoparametric,kimura1986real,park1989isoparametric,dominguez2016isoparametric}), complex hyperbolic spaces $\mC\mH^n$ (\cite{berndt1989real,berndt2006real,berndt2007real,diaz2012inhomogeneous,diaz2017isoparametric}), compact symmetric spaces (\cite{murphy2012curvature}), and the product space $\mS^2\times\mS^2$ (\cite{urbano2019hypersurfaces}).
\par
For general  Riemannian manifolds, a significant result is Wang's work \cite{isoparaWang1987} on the tubular regularity of transnormal functions on connected complete Riemannian manifolds,
showing that the focal varieties, which are the level sets taking the maximal or minimal values, are submanifolds, while the other level sets are tubes over either of the focal varieties. This result directly leads to the fact that an arbitrary transnormal function $f$ on $(M,g)$ can induce an embedded {\it transnormal system} (see \cite{bolton1973transnormal}) $\mathcal{F}_f$ of codimension 1. More precisely:
\begin{itemize}
\item $\mathcal{F}_f$ consists of all connected components of all level sets of $f$, which are all submanifolds of $M$, called {\it foils};
\item any geodesic in $M$ intersects these foils, orthogonally at all or none of its points;
\item each foil $L$ has to be a hypersurface of $M$ whenever $f(L)\notin \{\min f,\max f\}$;
\item all foils in $\mathcal{F}_f$ are embedded.
\end{itemize}

\par
Wang's regularity theorem also suggests that any transnormal function is a {\it Morse-Bott function}, i.e., a smooth function whose critical set is a union of closed submanifolds, with the Hessian non-degenerate in directions normal to each submanifold.
On the other hand, Qian and Tang (\cite{qian2015isoparametric}) proved the following result: if a connected compact manifold $M$ admit a Morse-Bott function $f$,
whose critical set is the union of two connected compact submanifolds of codimension greater than 1, then $M$ can be endowed with another Riemannian metric such that $f$ becomes an isoparametric function. This shows that, the existence of transnormal and isoparametric functions impose the same topological constraints on a specific type of manifolds as above.
\par
It is natural to ask the following two questions:
\begin{itemize}
\item Given an embedded transnormal system $\mathcal{F}$ on $(M,g)$ of codimension 1, whether we can find a transnormal function $f$,
such that the transnormal system $\mathcal{F}_f$ coincides with $\mathcal{F}$.
\item Under whether conditions we can construct a Riemannian metric $g$ on a given manifold $M$, so that $(M,g)$ admits a transnormal (or isoparametric) function.
\end{itemize}

To explore these problems, we first turn our attention to the geometric structure of transnormal systems on connected complete Riemannian manifolds.
As shown by Bolton (\cite{bolton1973transnormal}), in an embedded transnormal system of codimension 1,
at most two foils are {\it singular} (whose codimension are strictly greater than $1$)
and other foils are all {\it regular}. Focusing on the topology of normal sphere bundles,
the regular foils can be further classified, and hence we establish a theorem on the global geometric structure of transnormal systems,
 improving Bolton's result:
\begin{defi}
    Let $\F$ be an embedded transnormal system of codimension 1 on $M$ and $L\in \F$ be a regular foil.
    If the normal sphere bundle of $L$ (i.e. the collections of all unit normal vectors on $L$) is disconnected,
    we call $L$ a {\it double-sided foil} ({\it DR-foil}); otherwise, $L$ is a {\it single-sided foil} ({\it SR-foil}).
    Consistently, a singular foil is abbreviated as an {\it S-foil}.

\end{defi}
\begin{thmintro}\label{thm: intro: improving Bolton}
    Let $\F$ be an embedded transnormal system of codimension 1 on a complete connected Riemannian manifold $M$,
 $N_{\text{SR}}$ and $N_{\text{S}}$ be the numbers of SR-foils and S-foils in $\F$, respectively,
    $N_\text{C}:=N_{\text{SR}}+N_{\text{S}}$ be the number of foils whose normal sphere bundle is connected,
    and
    $$D:=\sup\{d(L_1,L_2):L_1,L_2\in \F\}$$
    be the {\it diameter} of $\F$.
    Then $N_C$ is at most two. More precisely:
    \begin{itemize}
    \item If $D$ is infinity, then $N_\text{C}=1$ or $0$ and $M$ is diffeomorphic to the normal bundle of a foil $L\in \F$,
    where $L$ is the unique SR-foil or S-foil whenever $N_\text{C}=1$, or can be taken to be an arbitrary DR-foil whenever $N_\text{C}=0$.
    \item If $D$ is finite, then $N_\text{C}=2$ or $0$ and $M$ is diffeomorphic to the union of two normal disk bundles over foils $L,L'\in \F$
    glued together with their common boundary, where $L,L'$ are the only two foils whose normal sphere bundle is connected whenever $N_\text{C}=2$,
    or can be taken to be any pair of DR-foils satisfying $d(L,L')=D$ whenever $N_\text{C}=0$.
    \end{itemize}

\end{thmintro}
\par
Meanwhile, we can construct transnormal functions on such manifolds, 
showing the equivalence between the existences of transnormal functions and embedded transnormal systems of codimension 1
for connected complete Riemannian manifolds:
\begin{thmintro}\label{thm: intro: tran sys imply trans func and two structures}
    Suppose $\F$ is an embedded  transnormal system of codimension 1 on a connected complete manifold $M$, then $M$ admits a transnormal function $f$ such that the transnormal system $\F_f$ induced by $f$ coincides with $\F$.
\end{thmintro}
\par
 Theorem \ref{thm: intro: tran sys imply trans func and two structures} has been presented in \cite{MIYAOKA2013130}.
But the proof in the present paper is more rigorous, covering more cases that can be easily ignored,
see Remark \ref{remark: omit} for details.


As shown in Theorem \ref{thm: intro: improving Bolton} and Theorem \ref{thm: intro: tran sys imply trans func and two structures}, when a connected complete Riemannian manifold $(M,g)$
admits a transnormal function, the underlying smooth manifold $M$ has either of the following structures:
\begin{enumerate}[(i)]
        \item $M$ is diffeomorphic to a vector bundle over a smooth manifold $N$;
        \item $M$ admits a {\it linear double disk bundle decomposition} ({\it LDDBD}), that is,
        there exist linear disk bundles $\mathcal{D}_1$ over $N_1$ and $\mathcal{D}_2$ over $N_2$, such that
        $M$ is diffeomorphic to the union of $\mathcal{D}_1$ and $\mathcal{D}_2$ glued together  along their boundary
        by a diffeomorphism $\phi:\partial \mathcal{D}_1\rightarrow \partial \mathcal{D}_2$.
    \end{enumerate}

On the other hand, for any vector bundle $M$ over $N$ equipped with an arbitrary bundle metric $\langle\cdot,\cdot\rangle$,
we can construct a Riemannian metric on this vector bundle with the aid of the linear connection associated with $\langle\cdot,\cdot\rangle$,
so that the squared norm function becomes an isoparametric function. Consequently, the connected components of the sphere bundles
as well as the base space constitute an embedded transnormal system $\mathcal{F}$, so that each regular foil has constant mean curvature.
 For any manifold $M$ admitting a LDDBD, the complement of the union of two base spaces $N_1,N_2$ is diffeomorphic to $\partial \mathcal{D}_1\times (-1,1)$;
 thereby, a smooth family of Riemannian metrics on $\partial\mathcal{D}_1$ gives rise to a Riemannian metric $g$ on $M$,
 so that $(M,g)$ admits an embedded transnormal system $\mathcal{F}$ consisting of all components of the sphere bundles over $N_i$ ($i=1,2$)
 and the base spaces. Under the additional condition that both $N_1$ and $N_2$ are compact manifolds,
 we can find a volume-preserving diffeomorphism $\hat{\phi}:\partial \mathcal{D}_1\rightarrow \partial \mathcal{D}_2$ that is isotopic to $\phi$,
 whose resulting manifold is diffeomorphic to $M$. Then, by adjusting the metrics on the DR-foils in $\mathcal{F}$,
 we can construct a new Riemannian metric $\tilde{g}$ on $M$ so that the normal exponential mapping between any two such foils  
  becomes a volume-perserving diffeomorphism up to a rescaling,
 which ensures all regular foils have constant mean curvature. This enable us to construct isoparametric functions on such manifolds,
 establishing the following conclusion:

\begin{thmintro}\label{thm: intro: two structures imply trans and isopara}
Let $M$ be a vector bundle or admit a LDDBD equipped with given bundle metrics, then there exist a Riemannian metric $g$ and a smooth function $f$ on $M$, such that
$f$ is a transnormal function on $(M,g)$, and the induced transnormal system $\mathcal{F}_f$ consists of the base spaces and the connected components of 
all subbundles with hypersphere fibers of either the vector bundle or the disk bundles.
Especially if $M$ is a vector bundle or a compact manifold, then $f$ becomes an isoparametric function on $(M,g)$; 
Moreover, if additionally the rank of each bundle is 1, i.e., $\mathcal{F}_f$ has no singular foils, then 
the isoparametric foliation is harmonic, i.e., every leaf is a minimal hypersurface.

\end{thmintro}
\par
This result is also a generalization of Theorem 1.1 in \cite{qian2015isoparametric}, which corresponds to the special case where $M$ admits a LDDBD and the fibers of both disk bundles have dimensions greater than 1. The proof in the present paper, based on the method utilized in \cite[Sec. 2]{qian2015isoparametric}, 
makes appreciable optimizations to accommodate the more general setting.
\par
Combining the above results, we learn that the existences of transnormal and isoparametric functions impose the same topological constraints in the following sense:
\begin{corointro}\label{coro: intro: trans to isopara}
    Let $f$ be a transnormal function on a complete Riemannian manifold $M$, such that every connected component of
    each level set is compact, then $M$ can be endowed with another Riemannian metric such that it admits an isoparametric function $\tilde{f}$
    which satisfies $\mathcal{F}_{\tilde{f}}=\mathcal{F}_{f}$.
\end{corointro}
\begin{corointro}\label{coro: intro: compact}
For each compact manifold $M$, the following statements are equivalent:
\begin{itemize}
\item $M$ admits a LDDBD.
\item $M$ can be endowed with a Riemannian metric so that it admits an embedded transnormal system of codimension 1.
\item $M$ can be endowed with a Riemannian metric so that it admits a transnormal function.
\item $M$ can be endowed with a Riemannian metric so that it admits an isoparametric function. 
\end{itemize}
\end{corointro}
\par
It is worth emphasizing that in Corollary \ref{coro: intro: trans to isopara}, sometimes $\tilde{f}$ cannot equal $f$ directly. A detailed discussion on this is provided in Remark \ref{rmk: trans never isopara}.
\par

\section{Global structure of transnormal systems}\label{sec: global struc of trans struc}

Let $(M,g)$ be a connected complete Riemannian manifold and $\mathcal{F}$ be an embedded transnormal system of codimension 1.
For an arbitrary foil $L\in \mathcal{F}$, denote by $\text{N}(L)$ the normal bundle of $L$  and by $\mathcal{P}_L:\text{N}(L)\to L$ the bundle projection.
Let
\begin{equation*}
\text{NS}(L):=\{V\in \text{N}(L):|V|=1\}
\end{equation*}
be the {\it normal sphere bundle}. For the matter of convenience, we denote the normal exponential map on $\text{N}(L)$ by
\begin{equation*}
    \exp_{L}(t,V):=\exp_{\mathcal{P}_L(V)}(tV), \quad (t\in\mR,\ V\in \text{NS}(L)),
\end{equation*}
and then
\begin{equation*}
\aligned
    \mathcal{T}_{t}(L):=&\{\exp_{L}(t,V): V\in \text{NS}(L)\},\\
    \mathcal{N}_{t}(L):=&\{\exp_L(s,V): V\in \text{NS}(L),|s|<t\}
    \endaligned
\end{equation*}
are the {\it $t$-tube} and {\it $t$-neighborhood} of $L$ in $M$ for any $t>0$, respectively.

\par
If $L$ is a regular foil, then each fiber of $\text{NS}(L)$ contains exactly two vectors and hence $\text{NS}(L)$ has at most two connected components.
If $\text{NS}(L)$ is disconnected, i.e. $L$ is a DR-foil, then each connected component determines a smooth non-zero normal vector field on $L$,
which implies $\text{NS}(L)$ is a trivial bundle. On the other hand, if $\text{NS}(L)$ is connected, i.e. $L$ is an SR-foil, then for any unit normal vector $V$,
the path connecting $V$ and $-V$ in $\text{NS}(L)$ yields a closed curve in $L$ such that the normal vector traveling along this path returns in the opposite direction,
which shows $\text{NS}(L)$ is a nontrivial bundle.

\par
Fixing an unit normal vector $V_0$ of $L$, we denote by $B(L)$ the connected component of $\text{NS}(L)$ containing $V_0$, then
\begin{equation*}
    \text{NS}(L)=\left\{
        \begin{aligned}
            &B(L)\sqcup -B(L),\quad & \text{if $L$ is a DR-foil};\\
            &B(L),\quad & \text{if $L$ is an SR-foil or S-foil}.
        \end{aligned}
    \right.
\end{equation*}
\par
Under the above notations, Bolton's conclusions in \cite{bolton1973transnormal} can be summarized as the following lemma.
\begin{lem}[\cite{bolton1973transnormal}]\label{lemma: basic facts in Bolton}
    Let $\F$ be an embedded transnormal system of codimension one on a connected complete Riemannian manifold $M$,
    then all foils in $\mathcal{F}$ are complete, and the normal exponential map at $L$ satisfies the properties as follows:
    \begin{enumerate}[(i)]
        \item 
        For each $t\in \mR$, $\exp_L(t,B(L)):=\{\exp_L(t,V):V\in B(L)\}$ is a foil in $\mathcal{F}$;
        \item For any $L'\in \F$, there exists $t\in \mR$, such that $L'=\exp_L(t,B(L))$;
        \item For any $t\neq 0$ and $V\in B(L)$, $(t,V)$ is a critical point of $\exp_L$ if and only if $\exp_L(t, B(L))$ is an S-foil.
    \end{enumerate}
\end{lem}

This enable us to derive the following lemma concerning the non-injectivity of the normal exponent map, which will be frequently used in subsequent proofs.
\begin{lem}\label{non-inj}
If $\exp_L(t_1,V_1)=q=\exp_L(t_2,V_2)$ lies in a regular foil, where $V_1,V_2\in B(L)$ and $(t_1,V_1)\neq (t_2,V_2)$, then one and only one of the following 2 cases must occur:
\begin{enumerate}[(i)]
\item $\exp_L(\frac{t_1+t_2}{2}+r,B(L))=\exp_L(\frac{t_1+t_2}{2}-r,B(L))$ holds for each $r\in \mR$ and $\exp_L(\frac{t_1+t_2}{2},B(L))$ is an SR-foil or an S-foil.
\item $t_1\neq t_2$ and $\exp_L(r,B(L))=\exp_L(r+t_2-t_1,B(L))$ for each $r\in \mR$.
\end{enumerate}
\end{lem}

\begin{proof}
Let $L'$ be the regular foil containing $q$. Due to the the definition of transnormal systems,
$$\gamma_1(s):=\exp_L(t_1+s,V_1)\text{ and }\gamma_2(s):=\exp_L(t_2+s,V_2)$$
are both arc-length parametrized geodesics orthogonal to $L'$. This implies either $\gamma_{1}(s)=\gamma_{2}(-s)$ or $\gamma_{1}(s)=\gamma_{2}(s)$.
If $\gamma_1(s)=\gamma_2(-s)$, letting $s:=\frac{t_2-t_1}{2}+r$ and then Lemma \ref{lemma: basic facts in Bolton} enable us to derive
$\exp_L(\frac{t_1+t_2}{2}+r,B(L))=\exp_L(\frac{t_1+t_2}{2}-r,B(L))$, which implies the connectedness of the normal sphere bundle of $\exp_L(\frac{t_1+t_2}{2},B(L))$.
If $\gamma_1(s)=\gamma_2(s)$, $(t_1,V_1)\neq (t_2,V_2)$ forces $t_1\neq t_2$, then we arrive at the case (ii) by putting $s:=r-t_1$.

\end{proof}

We aim to further discuss the cut locus of $L$. Let 
\begin{equation*}\label{equ: def of T}
    T:=\sup \big\{t\in(0,+\infty): \exp_L:(0,t)\times \text{NS}(L)\to \mathcal{N}_{t}(L)\backslash L\ \text{is a homeomorphism}\big\}
\end{equation*}
be the {\it injectivity radius} of $L$. 
We now present properties of the injectivity radius in connection with embedded transnormal systems of codimension 1.
\begin{lem}\label{lemma: summary of basic}
    For an arbitrary foil $L\in \F$ and the injecticity radius $T$ of $L$, we have:
    \begin{enumerate}[(i)]
        \item  $T>0$.
        \item For each $t$ satisfying $0<|t|<T$, $\exp_{L}(t, B(L))$ is a DR-foil.
        \item For any $t_0\in \mR$, $L':=\exp_{L}(t_0, B(L))$ being a DR-foil ensures the existence of $\delta>0$,
        such that $\exp_L$ is a homeomorphism between $(t_0-\delta,t_0+\delta)\times  B(L)$ and $\mathcal{N}_{\delta}(L')$.
        \item If $L$ is an SR-foil or S-foil and $T<\infty$, then $\exp_{L}(T,\text{NS}(L))$ is also an SR-foil or S-foil.
        \item If all foils in $\mathcal{F}$ are DR-foils, then $\exp_{L}(T, B(L))=\exp_{L}(-T, B(L))$ whenever $T<\infty$.
    \end{enumerate}
\end{lem}
\begin{proof}
        Due to Lemma \ref{lemma: basic facts in Bolton}, in conjunction with the fact that the normal exponential map $\exp_L$ is a local homeomorphism near any zero vector,
        we know $\exp_L$ has no critical point in $(0,\delta_0)\times \text{NS}(L)$ for a sufficient small positive number $\delta_0$. It suffices to show $\exp_L$ is injective on $(0,\delta)\times \text{NS}(L)$ for some $\delta\in(0,\delta_0]$. Assume by contradiction $\{V_n\},\{V'_n\}$ are sequences in $\text{NS}(L)$ such that $$\exp_L(\varepsilon_n,V_n)=q_n=\exp_L(\varepsilon'_n,V'_n)$$
        with $\lim\limits_{n\to \infty}\varepsilon_n=\lim\limits_{n\to \infty}\varepsilon'_n=0$ but $(\varepsilon_n,V_n)\neq (\varepsilon'_n,V'_n)$.
        Since $\text{NS}(L)$ has at most 2 connected components, without loss of generality we can assume all $V_n$'s are lying in 
        $B(L)$. 
        Applying Lemma \ref{non-inj} shows the existence of nonzero $\delta_n$ ($\delta_n=\varepsilon_n+\varepsilon'_n$ or $\varepsilon_n-\varepsilon'_n$),
        such that $\lim\limits_{n\to \infty}\delta_n=0$ and $\exp_L(\delta_n, B(L))=L$,
        which causes a contradiction to the embeddedness of $L$.
        This completes the proof of (i).

        For each $t$ satisfying $|t|\in (0,T)$, $\frac{d}{ds}\big|_{s=t}\exp_L(s,V)$ with all $V$'s in $B(L)$ form a smooth normal vector field on $L':=\exp_L(t, B(L))$,
        and hence $L'$ is a DR-foil. This completes the proof of (ii).


        In conjunction with (i), (ii) and Lemma \ref{lemma: basic facts in Bolton}, for each $L':=\exp_L(t_0,B(L))$,
        there exists a sufficiently small positive number $\delta_0$, such that $\exp_L(t, B(L))$ is a DR-foil in $\mathcal{N}_{\delta_0}(L')$ whenever $t\in(t_0-\delta_0,t_0+\delta_0)$,
        and hence $\exp_L$ has no critical point in $(t_0-\delta_0,t_0+\delta_0)\times B(L)$.
        Now we claim, $\exp_L$ is injective on $(t_0-\delta,t_0+\delta)\times  B(L)$ with $\delta:=\min\{\delta_0,\frac{T}{2}\}$.
        Assume by contradiction
        $$\exp_L(t_1,V_1)=q=\exp_L(t_2,V_2)$$
        for distinct $(t_1,V_1)$ and $(t_2,V_2)$ in this domain, satisfying $t_1\leq t_2$.
        Noting that $\exp_L(\frac{t_1+t_2}{2}, B(L))$ is a DR-foil, by applying Lemma \ref{non-inj} we have
        $$\exp_L(t_2-t_1,B(L))=L$$
        with $0<t_2-t_1<2\delta\leq T$,
        causing a contradiction to the definition of injectivity radius of $L$. This shows (iii).

        Assume $L$ is an SR-foil or S-foil and $T<\infty$.
        For any $t_0\in (0,T]$ so that $L':=\exp_L(t_0,\text{NS}(L))$ is a DR-foil, there exists a sufficiently small positive number $\delta_0$ given in (iii),
        such that $\exp_L(t,\text{NS}(L))$ is a DR-foil for each $t\in (t_0-\delta_0,t_0+\delta_0)$ and $\exp_L$ is injective on $(t_0-\delta_0,t_0+\delta_0)\times \text{NS}(L)$.
        We claim $\exp_L:(0,t_0+\delta)\times \text{NS}(L)\to\mathcal{N}_{t_0+\delta}(L)\backslash L $ is injective with $\delta:=\min\{\delta_0,T\}$.
        Otherwise, we can find $t_1\in(0,t_0-\delta]$, $t_2\in[t_0,t_0+\delta)$, such that
        $\exp_L(r,\text{NS}(L))=\exp_L(r+t_2-t_1,\text{NS}(L))$ for each $r\in \mR$.
        Letting $r:=-\frac{t_2-t_1}{2}+\varepsilon$ and then using $\exp_L(-t,\text{NS}(L))=\exp_L(t,\text{NS}(L))$, we arrive at
        $$\exp_L\big(\frac{t_2-t_1}{2}-\varepsilon,\text{NS}(L)\big)= \exp_L\big(-\frac{t_2-t_1}{2}+\varepsilon,\text{NS}(L)\big)=\exp_L\big(\frac{t_2-t_1}{2}+\varepsilon,\text{NS}(L)\big),$$
        which causes a contradiction when $|\varepsilon|$ is sufficiently small.
        Therefore 
         $\exp_L(T,\text{NS}(L))$ should be an SR-foil or S-foil. (iv) has been proved.


         Finally, assume all foils in $\mathcal{F}$ are DR-foils and $T<\infty$,
         we will show $L_1:=\exp_L(T, B(L))$ coincides with $L_2:=\exp_L(-T, B(L))$.
         If not, due to the embeddedness of $L_1$ and $L_2$, Lemma \ref{lemma: basic facts in Bolton} enable us to find $\delta_0\in (0,T)$, such that the $\delta_0$-neighborhoods of $L_1$ and $L_2$ are disjoint.
         we claim $\exp_L$ is injective on $(-T-\delta_0,T+\delta_0)\times  B(L)$ and obtain a contradiction to the definition of $T$.
         To prove this claim, we assume by contradiction that $\exp_L(t_1,V_1)=\exp_L(t_2,V_2)$ for $V_1,V_2\in B(L)$, $(t_1,V_1)\neq (t_2,V_2)$ and $|t_1|\geq T$,
         then $\mathcal{N}_{\delta_0}(L_1)\cap \mathcal{N}_{\delta_0}(L_2)=\emptyset$ implies $t_1\in (-T-\delta_0,-T]$, $t_2\in (-T+\delta_0,T-\delta_0)$,
         or $t_1\in [T,T+\delta_0)$, $t_2\in (-T-\delta_0,T+\delta_0)$. For either case, $0<|t_2-t_1|<2T$, and it follows from Lemma \ref{non-inj} that $\exp_L(-\frac{t_2-t_1}{2},B(L))=\exp_L(\frac{t_2-t_1}{2},B(L))$,
         causing a contradiction to the definition of injectivity radius.

\end{proof}
\par
Assume first $L$ is an SR-foil or an S-foil in $\mathcal{F}$ whenever $N_\text{C}\geq 1$.
We shall consider the injectivity radius $T$ of $L$ case by case:
\begin{itemize}
\item {\bf Case A. } If $T=+\infty$, the normal exponential map becomes a diffeomorphism between $\text{N}(L)$ and $M$,
and then all other foils in $\F$ must be OR-foils (by (ii) of Lemma \ref{lemma: summary of basic}), i.e. $N_{\text{C}}=1$. Therefore
$$\mathcal F=\{L_t:t\geq 0\}\qquad \text{with }L_t:=\exp_L(t, \text{NS}(L))$$
and
$$d(L_{t_1},L_{t_2})=|t_2-t_1|.$$
This implies $D=+\infty$. Denote by
$$d_L(x):=d(x,L)$$
the distance function from $L$,
then $f(x):=d_L^2(x)$ is a smooth function on $M$, which satisfies
$$x\in L_t\Longleftrightarrow f(x)=t^2$$
and
$$|\nabla f|^2=4f.$$
Hence $f$ is a transnormal function on $M$, such that the transnormal system $\mathcal{F}_f$ coincides with $\mathcal{F}$.

\item {\bf Case B.}
 If $T<+\infty$, by (iv) of Lemma \ref{lemma: summary of basic},
$L':=\exp_L(T,\text{NS}(L))$ is another OR-foil or S-foil. The connectivity of the tubes around $L$ and $L'$ implies
$$\exp_L(t,\text{NS}(L))=\exp_L(-t,\text{NS}(L)),\quad \exp_{L'}(t,\text{NS}(L'))=\exp_{L'}(-t,\text{NS}(L')).$$
In conjunction with (i) of Lemma \ref{lemma: basic facts in Bolton}, we have
$$\exp_L(T+t,\text{NS}(L))=\exp_{L'}(t,\text{NS}(L'))=\exp_{L'}(-t,\text{NS}(L'))=\exp_L(T-t,\text{NS}(L))$$
for each $t\in [0,T]$
and hence $\exp_L(2T,\text{NS}(L))=L$, moreover
$$\exp_L(2kT+t,\text{NS}(L))=\exp_L(t,\text{NS}(L))$$
for any $k\in \Bbb{Z}$ and $t\in [0,2T)$.
Along with (ii) of Lemma \ref{lemma: basic facts in Bolton}, we know
$$\mathcal F=\{L_t:t\in [0,T]\}\qquad \text{with }L_t:=\exp_L(t, \text{NS}(L))$$
and
$$d(L_{t_1},L_{t_2})=|t_2-t_1|.$$
This implies $D=T$ and $N_\text{C}=2$. In this case, $M$ is the union of $\overline{\mathcal{N}_t(L)}$ and $\overline{\mathcal{N}_{T-t}(L')}$
glued together with their common boundary, where $t$ can be taken by an arbitrary number in $(0,T)$.
Since $d_L(x)+d_{L'}(x)=T$,
$$f(x):=\cos\left(\frac{\pi}{T}d_L(x)\right)=-\cos\left(\frac{\pi}{T}d_{L'}(x)\right)$$
should be a smooth function on $M$, which satisfies
$$x\in L_t\Longleftrightarrow f(x)=\cos\left(\frac{\pi t}{T}\right)$$
and
$$|\nabla f|^2=\left(\frac{\pi}{T}\right)^2(1-f^2).$$
Hence $f$ is a transnormal function on $M$, such that $\mathcal{F}_f=\mathcal{F}$.

\end{itemize}

    \par
    When $\mathcal{F}$ contains no OR-foil or S-foil, fix an arbitrary foil $L\in\mathcal{F}$, and assume its injectivity radius is $T$.
Now we discuss on $T$ case by case as following:
    \begin{itemize}
    \item {\bf Case C. } If $T=+\infty$,
    the definition of injectivity radius ensures
 the normal exponential map becomes a diffeomorphism between $\text{N}(L)\cong L\times \mR$ and $M$. It follows that
 $$\mathcal F=\{L_t:t\in \mR\}\qquad \text{with }L_t:=\exp_L(t, B(L))$$
 and
 $$d(L_{t_1},L_{t_2})=|t_2-t_1|.$$
 Thus $D=+\infty$. Denote by $\hat{d}_L(x)$ be the oriented distance function from $L$, satisfying
 $$\hat{d}_L(x)=t\Longleftrightarrow x\in L_t,$$
 then $\hat{d}_L$ is a smooth function on $M$, and
 $$|\nabla \hat{d}_L|=1,$$
i.e., $f:=\hat{d}_L$ is a transnormal function, such that $\mathcal{F}_f=\mathcal{F}$.

\item {\bf Case D. } If $T<+\infty$, then (v) of Lemma \ref{lemma: summary of basic} gives $\exp_L(T,B(L))=\exp_L(-T,B(L))$.
By Lemma \ref{non-inj}, we have
$$\exp_L(t,B(L))=\exp_L(t+2T,B(L))\qquad \forall\, t\in \mR.$$
Therefore
$$\mathcal F=\{L_{[t]}:[t]\in \mR/2T\Bbb{Z}\}\qquad \text{with }L_{[t]}:=\exp_L(t, B(L))$$
and
$$d(L_{[t_1]},L_{[t_2]})=d([t_1],[t_2]):=\min\{|t_2+2kT-t_1|:k\in \Bbb{Z}\},$$
which implies $D=T$. In this case,
$M$ is the union of $\overline{\mathcal{N}_t(L)}$ and $\overline{\mathcal{N}_{T-t}(L')}$
glued together with their common boundary, where $L':=\exp_L(T,B(L))$ and $t$ can be taken by an arbitrary number in $(0,T)$.
Let $\hat{d}_L$ and $\hat{d}_{L'}$ be the oriented distance function from $L$ and $L'$, respectively, such that
$$\aligned
\hat{d}_L(x)=&t\in (-T,T]\Longleftrightarrow x\in L_{[t]},\\
 \hat{d}_{L'}(x)=&s\in (-T,T]\Longleftrightarrow x\in L_{[T-s]}.
 \endaligned$$
Then
$$f(x):=\sin\left(\frac{\pi\hat{d}_L(x)}{T}\right)=\sin\left(\frac{\pi\hat{d}_{L’}(x)}{T}\right)$$
is a smooth function on $M$, which satisfies
$$f(x)=u\Longleftrightarrow x\in L_{[t]}\cup L_{[T-t]}\quad\text{ with }t:=\frac{T}{\pi}\arcsin u$$
and
$$|\nabla f|^2=\left(\frac{\pi}{T}\right)^2(1-f^2).$$
In other words, $f$ is a transnormal function whose corresponding transnormal system is just $\mathcal{F}$
(here each level set of $f$, except for those corresponding to $\pm 1$, consists of exactly 2 foils in $\mathcal{F}$).

\end{itemize}

This completes the proof of Theorem \ref{thm: intro: improving Bolton} and Theorem \ref{thm: intro: tran sys imply trans func and two structures}.

\par
Furthermore, we can classify embedded transnormal systems of codimension 1 into the following 7 types according to the number of SR-foils and S-foils, as well as the diameter of $\mathcal{F}$. See Table~\ref{table: seven types} and Figure~\ref{fig: seven types} for a summary of these types.  
\begin{table}
    \centering
    \caption{Properties of the seven types.}
    \label{table: seven types}
    \renewcommand\arraystretch{1.2}
    \begin{tabular}{|c|c|c|c|c|}
    \noalign{\hrule height 1pt}
    Type & $N_{\text{SR}}$ & $N_{\text{S}}$ & Diameter of $\mathcal{F}$ & Metric structure of $\mathcal{F}$ \\ 
    \hline
    {\it Cylindrical} & 0               & 0              &  $+\infty$                &    $\mR$             \\ 
    \hline
    {\it Planar}      & 0               & 1              &  $+\infty$                &    $[0,+\infty)$     \\ 
    \hline
    {\it Twisted-cylindrical} & 1       & 0              &  $+\infty$                &    $[0,+\infty)$     \\ 
    \hline
    {\it Toric}       &  0              & 0              &  $T<+\infty$              &    $\mR/2T\Bbb{Z}$   \\ 
    \hline
    {\it Spherical}   &  0              & 2              &  $T<+\infty$              &    $[0,T]$           \\ 
    \hline
    {\it Real-projective}   &  1        & 1              &  $T<+\infty$              &    $[0,T]$           \\ 
    \hline
    {\it Klein-bottled}   &  2          & 0              &  $T<+\infty$              &    $[0,T]$           \\ 
    \noalign{\hrule height 1pt}
    \end{tabular}
\end{table}
\begin{figure}[htbp]
    \centering
    \includegraphics[width=0.95\textwidth]{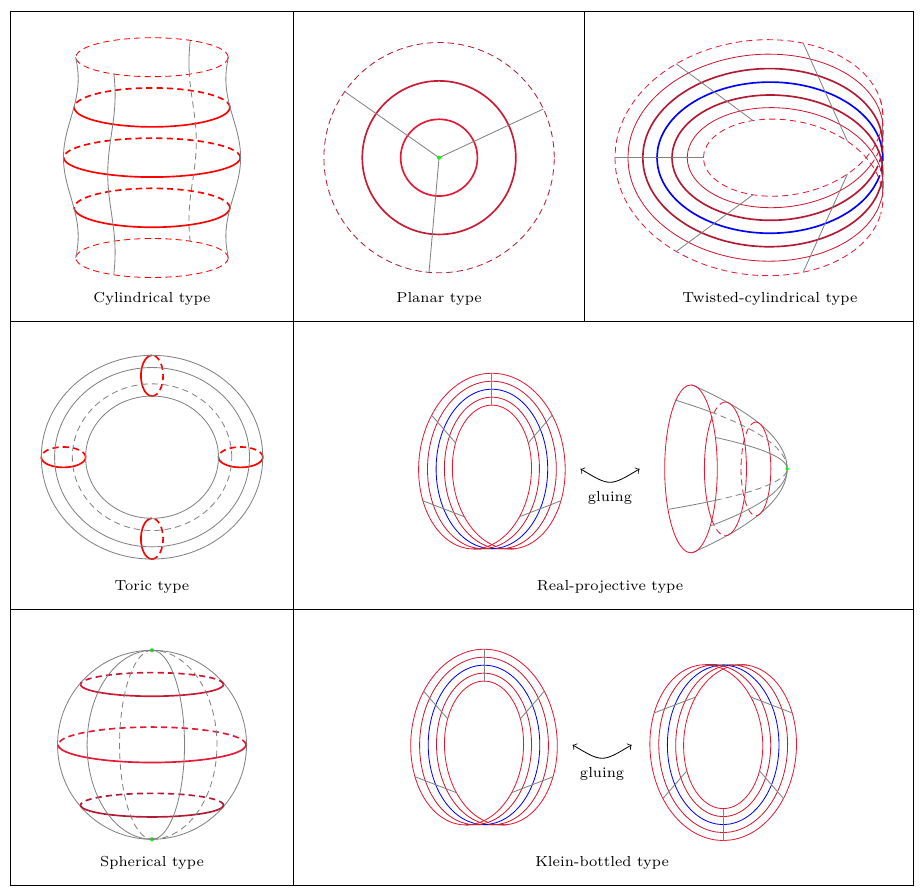}
    \caption{Illustrations of the seven types.}
    \label{fig: seven types}
    \vskip 0.2cm
    {\small Red indicates DR-foils, blue indicates SR-foils, green indicates S-foils, and gray represents normal geodesics.}
\end{figure}

\begin{rmk}\label{remark: omit}
When $\mathcal{F}$ is of Klein-bottled type, all foils are regular ones,
but the construction of the transnormal function is the same as in the spherical case.
Probably, this particular situation has been overlooked in some literature, e.g. \cite{MIYAOKA2013130}.

\end{rmk}

\section{The construction of transnormal and isoparametric functions}\label{sec: trans and isopara}

\par
The proof of Theorem \ref{thm: intro: two structures imply trans and isopara} relies on following lemmas concerning on the mean curvature
of regular level hypersurfaces of transnormal functions.
\par
\begin{lem}\label{lemma: mean curv and laplacian}
    Let $f$ be a transnormal function on a Riemannian manifold $M$ and $b$ be a smooth function, such that
    $|\nabla f|^2=b(f)$,
   then the Laplacian $\Delta f$ of $f$ on a regular level hypersurface $L$ is expressed as
    \[
        \Delta f=\frac12b'-H\cdot\sqrt{b},
    \]
    where $H$ denotes the mean curvature function on $L$ associated with the normal vector field $\frac{\nabla f}{|\nabla f|}$.
    Moreover, $f$ is isoparametric if and only if each regular level hypersurface (may be non-connected) has constant mean curvature.
\end{lem}

\begin{proof}
Without loss of generality we can assume $L$ is connected, then $L$ is a DR-foil in $\mathcal{F}_f$,
$B(L):=\{\frac{\nabla f}{|\nabla f|}(x):x\in L\}$ is a connected component of $\text{NS}(L)$, and there exists $\varepsilon>0$,
such that $L_t:=\exp_L(t,B(L))$ is a DR-foil for each $t\in (-\varepsilon,\varepsilon)$.
For any $x_0\in L$, there exists a local orthonormal frame $\{E_1,\cdots,E_{n-1},\nu\}$ on
$B_\delta(x_0):=\{y\in M: d(x_0,y)<\delta\}$ with $\delta\leq \varepsilon$, where $\nu:=\frac{\nabla f}{|\nabla f|}$
is the unit normal vector on the foil $L_t$ passing through the considered point,
then the mean curvature $H$ of $L$ at $x_0$ can be calculated as follows.
    \begin{align*}
        &H=-\langle\nabla_{E_i}\nu,E_i\rangle
        = -\frac{1}{|\nabla f|}\langle\nabla_{E_i}\nabla f,E_i\rangle\\
        =& -\frac{1}{|\nabla f|}\left(\Delta f-\langle\nabla_{\nu}\nabla f,\nu\rangle\right)
        =-\frac{\Delta f}{|\nabla f|}+\frac{\nu(|\nabla f|^2)}{2|\nabla f|^2}\\
        =&-\frac{\Delta f}{\sqrt{b}}+\frac{b'}{2\sqrt{b}}.
    \end{align*}
    This completes the proof of the present lemma.
\end{proof}

\par
\begin{lem}\label{lemma: mean cur and vol}
Let $L$ be an oriented hypersurface of $M$, $\nu$ be a unit normal vector field on $L$
and
$$\phi_t(x):=\exp_L(t,\nu)(x)=\exp_x(t\nu(x)) \qquad (\forall t\in \mR, x\in L)$$
be the normal exponential map from $L$.
Denote by $dV_t$
the volume form of $L_t:=\{\phi_t(x):x\in L\}$
(assume $L_t$ is a hypersurface for each $t\in (-\varepsilon,\varepsilon)$), such that
$$\phi_t^*dV_t=\lambda_tdV_0$$
then
$$\frac{d}{dt}\log\lambda_t(x)=-H_t(\phi_t(x))\qquad \forall x\in L$$
with $H_t$ the mean curvature function on $L_t$.

\end{lem}
\begin{rmk}
As a corollary, each $L_t$ has constant mean curvature, if and only if for each $t$,
$\phi_t:L\rightarrow L_t$ is a volume-preserving diffeomorphism up to a rescaling,
i.e. $\phi_t^*dV_t=\lambda_tdV_0$ with $\lambda_t$ only depending on $t$.
Furthermore, $\lambda_t\equiv 1$ for each $t$ ensures all of $L_t$ are minimal hypersurfaces.
If $L$ and $L_t$ are non-orientable, the volume forms $dV_0$ and $dV_t$ are understood in a local sense, i.e., as well-defined volume forms on a neighborhood of each considered point, and the same convention applies throughout the following discussion whenever a volume form on a non-orientable hypersurface is involved.
\end{rmk}

\begin{proof}
For each consider point $x\in L$, let $\{e_1,\cdots,e_{n-1}\}$ be an orthonormal basis of $T_x L$
and
$$g_{ij}(t):=\langle (\phi_t)_*e_i,(\phi_t)_*e_j\rangle,$$
then
$$\aligned
\lambda_t=&\lambda_t dV_0(e_1,\cdots,e_{n-1})
=\phi_t^*dV_t(e_1,\cdots,e_{n-1})\\
=&dV_t((\phi_t)_*e_1,\cdots,(\phi_t)_*e_{n-1})
=\det(g_{ij}(t))^{\frac{1}{2}}
\endaligned$$
and
$$\aligned
\frac{d}{dt}\Big|_{t=0}\lambda_t=&\frac{d}{dt}\Big|_{t=0}\det(g_{ij}(t))^{\frac{1}{2}}=\frac{1}{2}\frac{d}{dt}\Big|_{t=0}\det(g_{ij}(t))\\
=&\frac{1}{2}g'_{ii}(0)=\frac{1}{2}\nabla_{\frac{\partial}{\partial t}}\langle (\phi_t)_* e_i,(\phi_t)_* e_i\rangle\Big|_{t=0}\\
=&\langle \nabla_{e_i}\nu,e_i\rangle=-H_0(x).
\endaligned$$
For any $t_0\in (-\varepsilon,\varepsilon)$, denote by $\psi_s$ the normal exponential map from $L_{t_0}$,
then $\psi_s=\phi_{t_0+s}\circ \phi_{t_0}^{-1}$ and hence
$$\psi_s^*dV_{t_0+s}=\left(\frac{\lambda_{t_0+s}}{\lambda_{t_0}}\circ \phi_{t_0}^{-1}\right)dV_{t_0}.$$
Similarly as above, we have
$$H_{t_0}(\phi_{t_0}(x))=-\frac{d}{ds}\Big|_{s=0}\frac{\lambda_{t_0+s}}{\lambda_{t_0}}(x)=-\frac{d}{dt}\Big|_{t=t_0}\log \lambda_t(x).$$
\end{proof}
\par
Let $M$ be a vector bundle over $N$, $\pi:M\rightarrow N$ be the bundle projection.
Using local trivializations and a partition of unity, we can construct a {\it bundle metric} (see e.g. \cite[Theorem 2.1.4]{jost2008riemannian})
$\langle\cdot,\cdot\rangle$ and a linear connection $\overline{\nabla}$ compatible with $\langle\cdot,\cdot\rangle$. Namely,
let $s_1,s_2$ be smooth sections on an open domain $U$ of $N$ (that is, $s_i$ is a smooth mapping from $U$ into $M$ so that $\pi\circ s_i=\mathbf{Id}_U$),
$X$ and $f$ are tangent vector field and smooth function
on $U$, respectively, then
  \[
            \overline{\nabla}_X(s_1+f\,s_2)=\overline{\nabla}_X s_1+f\,\overline{\nabla}_X s_2+(Xf)s_2
        \]
and
 \[
            X\langle s_1,s_2\rangle=\langle \overline{\nabla}_X s_1,s_2\rangle+\langle s_1,\overline{\nabla}_X s_2\rangle.
        \]
For any smooth curve $\gamma: [a,b]\rightarrow N$ and $u\in \pi^{-1}(\gamma(a))$, there exists a unique
smooth curve $\xi:[a,b]\rightarrow M$, such that $\pi\circ \xi(t)=\gamma(t)$, $\xi(a)=u$
and $\overline{\nabla}_{\dot{\gamma}}\xi=0$. $\xi$ is said to be the {\it horizontal lift} of $\gamma$,
and $\dot{\xi}(a)$ is called a {horizontal vector} at $u$. All horizontal vectors form the {\it horizontal space} $\mathcal{H}_u$,
such that $\pi_*:\mathcal{H}_u\rightarrow T_{\pi(u)}N$ is a linear isomorphism.
It follows that
$$T_u M=\mathcal{H}_u\oplus \mathcal{V}_u$$
with $\mathcal{V}_u:=\{V\in T_u M:\pi_* V=0\}$ denoting the {\it vertical space} at $u$, i.e.
the tangent space of the fiber passing through $u$. Then
\begin{equation*}\label{metric-vb}
g_M(V_1,V_2):=g_N(\pi_* V_1,\pi_* V_2)+\langle V_1^\mathcal{V},V_2^\mathcal{V}\rangle
\end{equation*}
is an inner product on $T_u M$, where $V^{\mathcal{V}}_i$ is the vertical component of $V_i$ and $g_N$ is a fixed Riemannian metric on $N$.
Obviously $g_M$ is a Riemannian metric on $M$. Denote
$$f(u):=\langle u,u\rangle,$$
then $f$ is a smooth function on $M$. Let $\gamma:[a,b]\rightarrow N$ be a shortest geodesic on $(N,g_N)$,
$\xi:[a,b]\rightarrow M$ be the horizontal lift of $\gamma$ and $\eta:[a,b]\rightarrow M$ be another curve from
$\xi(a)$ to $\xi(b)$, then
$$\aligned
L(\eta)=&\int_a^b \sqrt{g_M(\dot{\eta},\dot{\eta})}dt=\int_a^b \sqrt{g_N(\pi_*\dot{\eta},\pi_* \dot{\eta})+\langle \dot{\eta}^\mathcal{V},\dot{\eta}^\mathcal{V}\rangle}\\
\geq& \int_a^b \sqrt{g_N((\pi\circ \eta)',(\pi\circ \eta)')}=L(\pi\circ \eta)\geq L(\gamma)
\endaligned$$
and the equality hold if and only if $\eta=\xi$ (up to a parameter transformation). This implies $\xi$ is a geodesic on $M$, and
$$\frac{d}{dt}\langle \xi(t),\xi(t)\rangle=2\langle \overline{\nabla}_{\dot{\gamma}}\xi,\xi\rangle=0.$$
Due to the arbitraries of $\gamma$, we have
$$\nabla f(W)=0,\qquad \text{Hess } f(W,W)=0$$
for any $W\in \mathcal{H}_u$. On the other hand, for each $V\in \mathcal{V}_u$, it is easily seen that
$t\in \mR \mapsto u+tV$ is a geodesic in $M$, hence
$$\nabla f(V)=\frac{d}{dt}\Big|_{t=0}\langle u+tV,u+tV\rangle=2\langle u,V\rangle$$
and
$$\text{Hess }f(V,V)=\frac{d^2}{dt^2}\Big|_{t=0}\langle u+tV,u+tV\rangle=2\langle V,V\rangle.$$
Therefore
$$|\nabla f|^2=4f,\qquad \Delta f=2k$$
with $k$ the rank of the vector bundle, i.e. the dimension of the vector fibers.
This means $f$ is not only transnormal, but also isoparametric.
For each $r\in (0,+\infty)$, let
$$\mathcal{S}(r):=\{u\in M:|u|=r\}$$
(where $|u|:=\langle u,u\rangle^{\frac{1}{2}}$) be the regular level hypersurface, then
$\mathcal{S}(r)$ is a subbundle of $M$ with hypersphere fibers, 
and for each $v\in \mathcal{S}(1)$,
$r\in \mR\mapsto rv$ is a arc-length parameterized geodesic which is orthogonal to all level sets.
Thus $\psi:(0,+\infty)\times \mathcal{S}(1)\rightarrow M\backslash N:=\{u\in M:\langle u,u\rangle>0\}$
$$(r,v)\mapsto rv$$
is a diffeomorphism, and the pull-back metric on $(0,+\infty)\times \mathcal{S}(1)$ can be written as
$$\psi^* g_M=dr^2+h(r).$$
On the other hand, due to Lemma \ref{lemma: mean curv and laplacian}, the mean curvature
$$H_r\equiv -\frac{k-1}{r}$$
on $\mathcal{S}(r)$ associated with the normal vectors pointing in the direction of the increasing of $f$.

\par
Now we consider a manifold $M$ admitting a LDDBD. Namely, there exist two vector bundles $M_1,M_2$ over $N_1,N_2$, equipped with
 bundle metrics $\langle\cdot,\cdot\rangle_1$ and $\langle\cdot,\cdot\rangle_2$, respectively; let
$$\mathcal{D}_i:=\{u_i\in M_i:|u_i|_i\leq 1\}\quad (i=1,2)$$
be the linear disk bundle, whose boundary $\partial \mathcal{D}_i$ is the sphere bundle;
then $M$ is the union of $\mathcal{D}_1$ and $\mathcal{D}_2$ glued along their boundary by a diffeomorphism $\phi:\partial \mathcal{D}_1\rightarrow \partial \mathcal{D}_2$.
Therefore,
$$M=W_1\cup W_2\cup W,$$
where
$$W_i:=\left\{u_i\in M_i:|u_i|_i<\frac{2}{3}\right\}\qquad (i=1,2)$$
and
$$W:=M\backslash \left(\left\{u_1\in M_1:|u_1|_1\leq \frac{1}{3}\right\}\cup \left\{u_2\in M_2:|u_2|_2<\frac{1}{3}\right\} \right)$$
are all open domains of $M$, and
$\psi:(-\frac{2}{3},\frac{2}{3})\times \partial \mathcal{D}_1\rightarrow W$
$$(r,v_1)\mapsto \begin{cases}
(r+1)v_1\in \mathcal{D}_1 & \text{when }r\leq 0\\
(1-r)\phi(v_1)\in \mathcal{D}_2 & \text{when }r>0
\end{cases}$$
is a diffeomorphism.

As shown above, there exists a Riemannian metric $g_i$ on $M_i$ ($i=1,2$), such that the squared norm function $f_i(u_i):=\langle u_i,u_i\rangle_i$
becomes a transnormal function, and
$$\psi_i^*g_i=dr_i^2+h_i(r_i)$$
with $\psi_i:(r_i,v_i)\mapsto r_iv_i$ the diffeomorphism between $(0,+\infty)\times \partial \mathcal{D}_i$ and $M_i\backslash N_i$
for $i=1,2$.
Now we define a Riemannian metric $g$ on $M$, such that
$$g|_{W_i}=g_i\qquad \forall i=1,2$$
and
$$\psi^*g|_W=dr^2+h(r),$$
where
$$h(r):=F(r)h_1(r+1)+(1-F(r))\phi^*h_2(1-r)$$
with $F(r)$ a smooth decreasing function on $(-\frac{2}{3},\frac{2}{3})$ satisfying
$$F|_{(-\frac{2}{3},-\frac{1}{6})}\equiv 1,\ F|_{(\frac{1}{6},\frac{2}{3})}\equiv 0.$$
(It is easy to check that $\psi^*g|_{W\cap W_i}=\psi^*g_i|_{W\cap W_i}$, hence $g$ is well-defined.)
Equipped with the above Riemannian metric,
$\psi(r,v_1)\in W\mapsto r$ is obviously a transnormal function, the length of whose gradient is $1$ everywhere.
Hence
$$\mathcal{F}:=\{L_1\in \mathcal{F}_{f_1}:L_1\subset \mathcal{D}_1\}\cup \{L_2\in \mathcal{F}_{f_2}:L_2\subset \mathcal{D}_2\}$$
(here $L_1$ and $\phi(L_1)$ are seen as the same foil whenever $L_1\subset \partial\mathcal{D}_1$)
is an embedded transnormal system of codimension 1 on $(M,g)$, and then Theorem \ref{thm: intro: tran sys imply trans func and two structures}
implies the existence of a transnormal function $f$ on $(M,g)$.

Next, under appropriate compactness conditions, we can refine this Riemannian metric so that
$f$ becomes an isoparametric function.

Now we assume both $N_1$ and $N_2$ are compact manifolds,
then $\partial \mathcal{D}_1$ and $\partial \mathcal{D}_2$ are also compact ones.
Due to the Moser's trick (see \cite{moser1965volume,dacorogna1990partial}),
we can find a diffeomorphism $\chi:\partial\mathcal{D}_1\rightarrow \partial\mathcal{D}_1$, such that
\begin{itemize}
\item $\chi$ is isotopic to the identity, i.e. there exist a 1-parameter family of diffeomorphisms $\{\chi_t: \partial\mathcal{D}_1\rightarrow \partial\mathcal{D}_1\}$
smoothly depending on $t$,
such that $\chi_0=\mathbf{Id}_{\partial\mathcal{D}_1}$, $\chi_1=\chi$;
\item $\chi:(\partial\mathcal{D}_1,h_1(1))\rightarrow (\partial\mathcal{D}_1,\phi^*h_2(1))$ is a volume-preserving map.
 i.e. $dV_1=\chi^*(\phi^*dV_2)$, where $dV_i$ is the volume form on $(U_i,h_i(1))$,
 $U_1$ is a neighborhood around each given $x\in \partial\mathcal{D}_1$ and $U_2:=(\phi\circ \chi)(U_1)$.
\end{itemize}
Thereby, $\hat{\phi}:=\phi\circ \chi:(\partial\mathcal{D}_1,h_1(1))\rightarrow (\partial\mathcal{D}_2,h_2(1))$
is a volume-preserving diffeomorphism that is isotopic to $\phi$. The attached mapping $\hat{\phi}$
can also result in a smooth manifold, denoted by $\hat{M}$. Using the isotopy between $\phi$ and $\hat{\phi}$,
it is easy for us to construct a diffeomorphism between $M$ and $\hat{M}$ that preserves the subbundles with hypersphere fibers.
Therefore, without loss of generality, we can assume $\hat{M}=M$, i.e. $\phi$ is a volume-preserving diffeomorphism.

Denote by
$$\mathcal{S}_i(r):=\{u_i\in M_i:|u_i|_i=r\}\qquad (\forall r\in (0,+\infty))$$
the sphere subbundle of $M_i$ ($i=1,2$),
then $\phi_{i,t}:\partial \mathcal{D}_i\rightarrow \mathcal{S}_i(1+t)$
$$v_i\mapsto \psi_i(1+t,v_i)$$
is the normal exponential map from $\partial\mathcal{D}_i$.
By Lemma \ref{lemma: mean cur and vol}, a straightforward calculation shows
$$dV_i(r_i)=\phi_{i,r_i-1}^*dV_i=r_i^{k_i-1}dV_i.$$
Here $k_i$ is the rank of $M_i$ and $dV_i(r_i)$ denotes the volume form of $(\partial\mathcal{D}_i,h_i(r_i))$. 
Let $dV(r)$ be the volume form of $(\partial\mathcal{D}_1, h(r))$,
then
$$dV(r)=dV_1(r+1)=(r+1)^{k_1-1}dV_1\qquad \forall r\in (-\frac{2}{3},-\frac{1}{6})$$
and
$$dV(r)=\phi^*dV_2(1-r)=(1-r)^{k_2-1}\phi^*dV_2=(1-r)^{k_2-1}dV_1\qquad \forall r\in (\frac{1}{6},\frac{2}{3}).$$
Now we fix a positive smooth function $\Phi(r)$ on $(-\frac{2}{3},\frac{2}{3})$,
such that
$$\Phi(r)=\begin{cases}
(r+1)^{k_1-1} & \text{when }r\in (-\frac{2}{3},-\frac{1}{6}),\\
(1-r)^{k_2-1} & \text{when }r\in (\frac{1}{6},\frac{2}{3}).
\end{cases}$$
then there exists a positive function $\mu(r,x)$, such that
$$dV(r)=\mu^{-(n-1)}(r,x)\Phi(r)dV_1.$$
Let $\tilde{g}$ be a new Riemannian metric on $M$, such that $\tilde{g}|_{W_i}=g_i|_{W_i}$ for $i=1,2$
and
$$\psi^*\tilde{g}|_W=dr^2+\mu^2(r,x)h(r).$$
Since
$$d\tilde{V}(r)=\Phi(r)dV_1$$
with $d\tilde{V}(r)$ the volume form associated with $(\partial{\mathcal{D}}_1,\mu^2(r,x)h(r))$,
again applying Lemma \ref{lemma: mean cur and vol}, we see each DR-foil in the transnormal
system $\mathcal{F}$ has constant mean curvature. Particularly if $k_1=k_2=1$, i.e.
all foils in $\mathcal{F}$ are regular, we can take $\Phi(r)\equiv 1$ so that
each foil is a minimal hypersurface. Finally Lemma \ref{lemma: mean curv and laplacian} ensures
$f$ is an isoparametric function on $(M,\tilde{g})$. This completes the proof of Theorem \ref{thm: intro: two structures imply trans and isopara}.
Combining Theorem \ref{thm: intro: two structures imply trans and isopara}, Wang's regularity theorem \cite{isoparaWang1987} and
Theorem \ref{thm: intro: tran sys imply trans func and two structures},  we can immediately derive Corollary \ref{coro: intro: trans to isopara}
and Corollary \ref{coro: intro: compact}.

\par
\begin{rmk}\label{rmk: trans never isopara}
    Sometimes, the transnormal function $f$ in Corollary \ref{coro: intro: trans to isopara} will never become isoparametric under any Riemannian metric. For instance, the radial
     $$f:=\cos r=\cos\left(\sqrt{x_1^2+\cdots+x_n^2}\right)$$ 
     on the Euclidean space $\mR^n$ $(n\ge 2)$ is a transnormal function that induces a transnormal system of planar type. 
     Assume $f$ is an isoparametric function on $(\mR^n,g)$, then Lemma \ref{lemma: mean curv and laplacian} implies
     each $S_r:=\{(x_1,\cdots,x_n)\in \mR^n:\sqrt{x_1^2+\cdots+x_n^2}=r\}$ with $r>0$ has constant mean curvature $H_r$,
     and $H_{r_1}=H_{r_2}$ whenever $\cos(r_1)=\cos(r_2)$. Thus
     $$H_{2\pi}=\lim_{t\rightarrow 0^+}H_{2\pi+t}=\lim_{t\rightarrow 0^+}H_{t}=+\infty$$
     and causes a contradiction.
     In contrast, another radial function $\tilde{f}:=r^2$ is an isoparametric function, and induces the same transnormal system as $f$. Moreover, if each level set of $f$ is connected, then we can take the function $\tilde{f}$ in Corollary \ref{coro: intro: trans to isopara} to be $f$.
\end{rmk}
\begin{rmk}
    Isoparametric functions on certain compact simply connected manifolds have been extensively studied in \cite{ge2013isoparametric,ge2015differentiable,qian2015isoparametric}. In these cases, attention can be restricted to the so-called {\it properly isoparametric} functions, which are defined in \cite{ge2013isoparametric} as isoparametric functions with focal varieties of codimension greater than one. A notable property of properly isoparametric functions is that each level set is connected, which significantly simplifies the discussion of the structure of the underlying manifold. However, this property does not hold in general; indeed, for any transnormal system $\F$ of toric type, 
    since the metric space $(\F,d)$ is homeomorphic to $\mS^1$, the transnormal function inducing $\F$ also induces a continuous function from $\mS^1$ to $\mR$, which cannot be injective;
    consequently, some level sets of such functions has to be disconnected. In fact, the absence of this significant property partially explains why, in the general case, 
    considering the associated transnormal system can be more effective than directly considering the transnormal function as in the compact simply connected case.

\end{rmk}

\end{document}